\documentclass[10pt]{article}
\usepackage{amsmath}
\usepackage{amsfonts}
\usepackage{amsthm}
\usepackage{amssymb}

\def\essinf{\mathop{\textrm{ess\,inf}}}
\def\esssup{\mathop{\textrm{ess\,sup}}}

\def\pO{\partial\Omega}

\def\divv{\mathrm{div}\,}

\newtheorem{theorem}{Theorem}[section]

\newtheorem{lemma}[theorem]{Lemma}
\newtheorem{remark}[theorem]{Remark}

\newtheorem{definition}[theorem]{Definition}
\newtheorem{example}[theorem]{Example}
\numberwithin{equation}{section}

\title{S+ property for divergence operator $-\divv a(z, u(z), \nabla u(z))$ with $p(\cdot)$ - growth condition and applications}

\author{
  Krzysztof Winowski\footnote{.}\\
  Jagiellonian University,\\
  Faculty of Mathematics and Computer Science,\\
  ul. {\L}ojasiewicza 6,\
  30-348 Krak\'ow, Poland\\
  email: krzysztof.winowski@im.uj.edu.pl\\
\\
}

\begin{document}
\maketitle
\begin{abstract}
We establish a result which provides the criteria for a divergence operator dependent both on the gradient and the function itself, to satisfy the S+ property. The result is complemented by several examples and an application.
\end{abstract}
\section{Introduction}
We begin the paper from reminding the definition of S+ property. 
\begin{definition}[S+ property]
	Let $X$ be a normed space and $A: X \to X^{*}$ be an operator. We say that the operator $A$ has S+ property if for every sequence $\{x_{n}\}_{n \in \mathbb{N}} \subset X$ and for every $x \in  X$ if
	$$x_{n} \rightharpoonup x \quad \textrm{weakly in} \quad X$$
	and
	$$\limsup_{n \to +\infty} \langle A(x_{n}), x_{n} - x\rangle \le 0$$ 
	then
	$$x_{n} \to x \quad \textrm{strongly in} \quad X.$$
\end{definition}
The above property is a valuable tool in analysis of boundary value problems. We continue by presenting some of its applications.
 
    Consider the following problem (see for example  \cite{Zeidler}, p. 589), where $X$ is a real separable and reflexive Banach space with dim X = $+\infty$ and $A: X \to X^{*}$ is an operator. Assume that $b \in X^{*}$ and consider the following abstract equation
$$Au = b.$$
Assuming that the operator $A$ is bounded, pseudomonotone and coercive (see Definitions \ref{Appendix def 1}, \ref{Appendix def 2}, \ref{Appendix def 3} in Appendix) and
using the Brezis theorem (see \cite{Zeidler}, p. 589) one can boost the weak convergence of the  Galerkin method to get the strong convergence (see \cite{Zeidler}, p. 589) under the additional assumption that the operator $A$ has S+ property. 

   	In \cite{Gas Papa 1} Gasi\'nski and Papageorgiou consider the following Robin boundary value problem
\begin{equation}\label{I eq1}
P(\lambda)\quad \left\{
\begin{array}{l}
-\Delta_{p}u(z) = \lambda f(z, u(z)) \quad \textrm{in} \quad \Omega,
\\
\displaystyle
\frac{\partial u}{\partial n_{p}} + \beta(z)u^{p - 1} = 0 \quad \textrm{on} \quad \partial \Omega,
\\
\lambda > 0, u \ge 0, 1 < p < +\infty.
\end{array}
\right.
\end{equation}
They define the set $L = \{\lambda > 0:\quad \textrm{problem}\quad P(\lambda)\quad \textrm{has a positive solution}\}$ (see p. 14) and under the certain assumptions on the set $\Omega$ and  the functions $f$ and $\beta$ they prove in Proposition 4.4 that  $L \ne \emptyset$. They define  $\lambda^{*} = \sup L$ (see p. 20) and prove in Proposition 4.8 that $\lambda^{*} < +\infty$. Using the S+ property of the operator
 $A: W^{1, p}(\Omega) \to W^{1, p}(\Omega)^{*}$ defined by
 $$\langle A(u), v \rangle = \int_{\Omega}|\nabla u(z)|^{p - 2}(\nabla u(z), \nabla v(z))_{\mathbb{R}^{N}}dz \quad \forall u, v \in W^{1, p}(\Omega)$$
  they show in Proposition 4.9 that $\lambda^{*} \in L$. Moreover using the same S+ property for the operator $A$ they prove in Proposition 4.10 that the problem $P(\lambda)$ has the smallest solution.

   	In \cite{Napoli Mariani} Napoli and Mariani consider the following Dirichlet problem
\begin{equation}\label{I eq2}
\left\{
\begin{array}{l}
-\divv a(z, \nabla u(z)) =  f(z, u(z)) \quad \textrm{in} \quad \Omega
\\
\displaystyle
x = 0 \quad \textrm{on} \quad \partial \Omega.
\end{array}
\right.
\end{equation}
Under the certain conditions on the set $\Omega$ and the functions $a$ and $f$ they prove in Theorem 3.1 that problem (\ref{I eq2}) has at least one nontrivial weak solution using the Mountain Pass Theorem. An important step in their proof was to consider the energy functional $J$ of this problem 
and to check in Lemma 3.1 that the functional $J$ satisfies the Palais-Smale condition. To do this step they use the S+ property for the Fr\'{e}chet derivative $J'$ (see Proposition 2.1).

    In \cite{Gas Papa} Gasi\'nski and Papageorgiou consider the following hemivariational inequality
\begin{equation}\label{I eq3}
\left\{
\begin{array}{l}
-\mathrm{div}\, a(z, u(z), \nabla u(z))\ - \partial j(z, u(z)) \ni\ f(z, u(z), \nabla u(z))
\quad\textrm{a.e. on}\ \Omega,\\
\displaystyle
x = 0
\quad\textrm{on}\ \partial \Omega.
\end{array}
\right.
\end{equation} 
They prove in Theorem 4.4.1 that under the certain assumptions on the set $\Omega$ and the functions $a$, $j$ and $f$ problem (\ref{I eq3}) has a solution $x \in W^{1, p}_{0}(\Omega)$. They introduce (see p. 561, 567, 568) the following three multivalued operators (see Appendix for the explanation of the notation)
$$V: W_{0}^{1, p}(\Omega) \to P_{wkc}(W^{-1, p'}(\Omega))$$
$$V(x) = \{-\divv v: v \in S^{p'}_{a(\cdot), x(\cdot), \nabla x(\cdot)}\} \quad \forall u \in W^{1, p}_{0}(\Omega)$$
$$N_{f}: W^{1, p}_{0}(\Omega) \to L^{p'}(\Omega)$$ 
$$N_{f}(x) = f(\cdot, x(\cdot), \nabla x(\cdot)) \quad \forall u \in W^{1, p}_{0}(\Omega)$$
$$G: W^{1, p}_{0}(\Omega) \to P_{wkc}(W^{-1, p^{'}}(\Omega))$$ 
$$G(x) = S^{r'}_{\partial j(\cdot), x(\cdot)} \quad \forall u \in W^{1, p}_{0}(\Omega).$$
Using these operators they define the multivalued operator $R: W^{1, p}_{0}(\Omega) \to P_{wkc}(W^{-1, p'}(\Omega))$ given by
$$R(u) = V(u) - G(u) - N_{f}(u)\quad \forall u \in W^{1, p}_{0}(\Omega).$$
They prove in Proposition 4.4.1 that the operator $V$ has S+ property and used this fact to prove that the operator $R$ is generalized pseudomonotone which is an important step in their proof of solution existence.

     We deal with the case of the divergence type operators in the case of space dependent exponent. Let us consider the operator
$A:  W^{1, p(\cdot)}(\Omega) \to  W^{1, p(\cdot)}(\Omega)^{*}$ defined by
$$\langle A(u), v \rangle = \int_{\Omega}(a(z, u(z), \nabla u(z)), \nabla v(z))_{\mathbb{R}^{N}}dz \quad \textrm{for all} \quad u, v \in W^{1, p(\cdot)}(\Omega).$$
   The main aim of this paper is to give assumptions which guarantee that the operator $A$ has S+  property. Such operator $A$ is a natural generalization of many operators arising in analysis of boundary value problems (see section 5). The main difficulty in the proof lies in the fact that the operator $A$ is not monotone.
   Having S+ property for the operator $A$ with $p(\cdot)$-growth we will be able to get the strong convergence in the Galerkin method and strenghten the result from \cite{Roubicek}. 
   
   Structure of the work is following. In section 2 we introduce notation and some theory on $L^{p(\cdot)}$ and $W^{1, p(\cdot)}$ spaces. In section 3 we prove some auxillary lemmas. In section 4 we state and prove the main result. In section 5 we present examples of operators which satisfy S+ property. In section 6 we give an application of S+. In section 7 (appendix) we collect some definitions from multivalued analysis. 
\section{Notation and preliminaries}
Throughout this paper we introduce the following notation. By
$\mathbb{N}_{0} = \{0, 1, 2, ...\}$ we denote the set of all natural numbers,
$\mathbb{N} = \{1, 2,...\}$ the set of all natural numbers without zero and
$N \in \mathbb{N}$.
For a normed space $X$ we denote by $X^{*}$ its topological dual.
By $\langle \cdot, \cdot \rangle$ we denote the duality brackets for the pair $(X^{*}, X)$.
By $(\cdot, \cdot)_{\mathbb{R}^{N}}$ we denote the standard scalar product in $\mathbb{R}^{N}$. By $|\cdot|$ we denote module in $\mathbb{R}$ or norm in $\mathbb{R}^{N}$. By $\|\cdot\|_{p(\cdot)}$ we denote norm in $L^{p(\cdot)}(\Omega)$ or in $L^{p(\cdot)}(\Omega; \mathbb{R}^{N})$.
By $\Omega$ we denote an open set in $\mathbb{R}^{N}$.
If we consider set $\Omega$ as a measure space we always consider it as a triple $(\Omega, \Sigma, \mu)$ where $\Sigma$ is the $\sigma$-algebra of Lebesgue measurable sets and $\mu$ is the Lebesgue measure.

Now we remind some preliminary facts and definitions from theory of $L^{p(\cdot)}$ and $W^{1, p(\cdot)}$ spaces.
\begin{definition}\label{N p def 1}
	  We define $P(\Omega)$ to be the set of all $\mu$-measurable functions $p: \Omega \to [1, +\infty]$. We define $L^{0}(\Omega)$ to be a real vector space of all $\mu$-measurable real valued functions. For function $p \in P(\Omega)$ we define $p_{-} = \essinf_{z \in \Omega}p(z)$ and $p^{+} = \esssup_{z \in \Omega}p(z)$. We say that functions $p, q \in P(\Omega)$ are conjugate if
	$$\frac{1}{p(z)} + \frac{1}{q(z)} = 1$$
	for almost all $z \in \Omega.$
\end{definition}

\begin{definition}\label{N p def 2}[variable exponent Lebesgue space](see Definition 3.2.1 in \cite{Di Ha Ha Ru})
	Let us fix a function $p \in P(\Omega)$, then we can consider the semimodular
	$$\rho_{p(\cdot)}(u) = \int_{\Omega}|u(z)|^{p(z)}dz \quad \textrm{for all} \quad u \in L^{0}(\Omega)$$
	on the real vector space $L^{0}(\Omega)$. We denote $\rho_{L^{p(\cdot)}(\Omega)} = \rho_{p(\cdot)}$. We define variable exponent Lebesgue space
	$$L^{p(\cdot)}(\Omega) = \{u \in L^{0}(\Omega):\quad \lim_{\lambda \to 0} \rho_{L^{p(\cdot)}(\Omega)}(\lambda u) = 0\}$$
	or equivalently
	$$L^{p(\cdot)}(\Omega) = \{u \in L^{0}(\Omega):\quad \rho_{L^{p(\cdot)}(\Omega)}(\lambda u) < +\infty \quad \textrm{for some}\quad \lambda > 0\}$$
	Then we can equip the space $L^{p(\cdot)}(\Omega)$ with the Luxemburg norm 
	$$\|u\|_{L^{p(\cdot)}(\Omega)} = \inf\left\{\lambda > 0: \rho_{L^{p(\cdot)}(\Omega)}\left(\frac{u}{\lambda}\right) \le 1\right\}.$$
	In the similar way we define also the normed space $L^{p(\cdot)}(\Omega; \mathbb{R}^{N})$.
\end{definition}

\begin{theorem}\label{N p theorem 1}[H\"{o}lder's inequality](see in \cite{Di Ha Ha Ru} Lemma 3.2.20)
	Let functions $p, q \in P(\Omega)$ be conjugate, then for all functions $u \in L^{p(\cdot)}(\Omega)$, $v \in L^{q(\cdot)}(\Omega)$ we have
	$$\int_{\Omega}|u||v|dz \le 2\|u\|_{p(\cdot)}\|v\|_{q(\cdot)}.$$
\end{theorem}

\begin{definition}\label{N p def 3}(see in \cite{Di Ha Ha Ru} Definition 2.5.1)
	For function $p \in P(\Omega)$ let us define the set
	$$E^{p(\cdot)}(\Omega) = \{u \in L^{p(\cdot)}(\Omega): \rho_{p(\cdot)}(\lambda u) < \infty \quad \textrm{for all} \quad \lambda > 0\}.$$
\end{definition}
\begin{theorem}\label{N p theorem 2}(see in \cite{Di Ha Ha Ru} Theorem 3.4.1)
	Let function $p \in P(\Omega)$, then the following conditions are equivalent:
	\item[(a)] $E^{p(\cdot)}(\Omega) = L^{p(\cdot)}(\Omega)$
	\item[(b)] $p^{+} < +\infty$.
\end{theorem}
\begin{theorem}\label{N p theorem 3}[Lebesgue dominated convergence theorem](see in \cite{Di Ha Ha Ru} Lemma 3.2.8)
	Let function $p \in P(\Omega)$, sequence $\{u_{n}\}_{n \in \mathbb{N}} \subset L^{0}(\Omega)$ and functions $u, v \in L^{0}(\Omega)$. If 
	$$u_{n}(z) \to u(z) \quad \textrm{for almost all}\quad z \in \Omega$$ 
	$$|u_{n}(z)| \le |v(z)|\quad \textrm{for almost all}\quad z \in \Omega\quad \textrm{and}\quad \forall n \in \mathbb{N}$$
	and
	$$v \in E^{p(\cdot)}(\Omega)$$
	then
	$$u_{n} \to u \quad \textrm{strongly in}\quad L^{p(\cdot)}(\Omega).$$ 
\end{theorem}
\begin{theorem}\label{N p theorem 4}
	Let function $p \in P(\Omega)$ with $p^{+} < +\infty$. Let us fix a sequence $\{u_{n}\}_{n \in \mathbb{N}} \subset L^{p(\cdot)}(\Omega)$ and a function $u \in L^{p(\cdot)}(\Omega)$. If 
	$$u_{n} \to u \quad \textrm{strongly in}\quad  L^{p(\cdot)}(\Omega)$$
	then we can take a subsequence $\{u_{n_{k}}\}_{k \in \mathbb{N}}$ and find a function $v \in L^{p(\cdot)}(\Omega)$ such that
	$$u_{n_{k}}(z) \to u(z) \quad \textrm{for almost all}\quad z \in \Omega$$
	and
	$$|u_{n_{k}}(z)| \le |v(z)| \quad \textrm{for almost all}\quad z \in \Omega \quad \textrm{and} \quad \forall k \in \mathbb{N}.$$ 
\end{theorem}
\begin{proof}
	The proof is analogous as in \cite{Kufner} Theorem 2.8.1.
\end{proof}
\begin{definition}[uniformly integrable sequence of functions]
	Let us fix a sequence of functions $\{u\}_{n \in \mathbb{N}} \subset L^{p(\cdot)}(\Omega)$. We say that this sequence is $p(\cdot)$-uniformly integrable if
		$$\forall \epsilon > 0 \quad \exists \delta(\epsilon) > 0 \quad \forall E \in \Sigma \quad \mu(E) < \delta(\epsilon) \quad \forall n \in \mathbb{N} \quad \int_{E}|u_{n}(z)|^{p(z)}dz < \epsilon.$$
\end{definition}
\begin{theorem}\label{N p theorem 5}[Vitali convergence theorem]
	Let function $p \in P(\Omega)$ with $p^{+} < +\infty$ and $\mu(\Omega) < +\infty$. Let us take a sequence $\{u_{n}\}_{n \in \mathbb{N}} \subset L^{p(\cdot)}(\Omega)$ and a function $u \in L^{p(\cdot)}(\Omega)$. If the family of functions $\{u_{n}\}_{n \in \mathbb{N}}$ is $p(\cdot)$-uniformly integrable and
	$$u_{n}(z) \to u(z) \quad \textrm{for almost all} \quad z \in \Omega$$
 then
	$$u_{n} \to u \quad \textrm{in}\quad L^{p(\cdot)}(\Omega).$$
\end{theorem}
\begin{proof}
	In view of Lemma \ref{N p lemma 1} it is enough to prove that $\int_{\Omega}|u_{n} - u|^{p(\cdot)}d\mu \to 0$. Now the proof is analogous as proof of Vitali converegnce theorem for $L^{p}$ spaces.
\end{proof}
\begin{remark}
	H\"{o}lder's inequality, Lebesgue dominated convergence theorem, Theorem \ref{N p theorem 4} and Vitali convergence theorem also hold for spaces $L^{p(\cdot)}(\Omega;\mathbb{R}^{N})$.
\end{remark}
\begin{lemma}\label{N p lemma 1}(see \cite{Fan Zhang} Proposition 2.3)
	Let us take a function $p \in P(\Omega)$ with $p^{+} < +\infty$, a sequence $\{u_{n}\}_{n \in \mathbb{N}} \subset L^{p(\cdot)}(\Omega)$ and a function $u \in L^{p(\cdot)}(\Omega)$, then we have
	\begin{equation}
	\|u\|_{p(\cdot)} \ge 1 \Rightarrow \|u\|^{p^{-}}_{p(\cdot)} \le \rho_{p(\cdot)}(u) \le \|u\|^{p^{+}}_{p(\cdot)}
	\end{equation}
	\begin{equation}
	\|u\|_{p(\cdot)} < 1 \Rightarrow \|u\|^{p^{+}}_{p(\cdot)} \le \rho_{p(\cdot)}(u) \le \|u\|^{p^{-}}_{p(\cdot)}
	\end{equation}
	\begin{equation}
	\|u_{n} - u\|_{p(\cdot)} \to 0 \Leftrightarrow \rho_{p(\cdot)}(u_{n} - u) \to 0.
	\end{equation}
\end{lemma}

\begin{definition}\label{N p theorem 6}[variable exponent Sobolev space](see \cite{Di Ha Ha Ru} Definition 8.2.1)
	Let $k \in \mathbb{N}_{0}$ and $p \in P(\Omega)$. The function $u \in L^{p(\cdot)}(\Omega)$ belongs to the space $W^{k, p(\cdot)}(\Omega)$, if its weak derivatives $\partial_{\alpha}u$ with $|\alpha| \le k$ exist and belong to $L^{p(\cdot)}(\Omega)$. We define a semimodular on $W^{k, p(\cdot)}(\Omega)$ by
	$$\rho_{W^{k, p(\cdot)}(\Omega)}(u) = \sum_{0 \le |\alpha| \le k} \rho_{L^{p(\cdot)}(\Omega)}(\partial_{\alpha}u)$$
	which induces a norm by
	$$\|u\|_{W^{k, p(\cdot)}(\Omega)} = \inf\left\{\lambda > 0: \rho_{W^{k, p(\cdot)}(\Omega)}\left(\frac{u}{\lambda}\right) \le 1\right\}.$$
\end{definition}
\begin{definition}(see in \cite{Fan Zhang} Proposition 2.5)
	For a function $p \in C(\overline{\Omega})$ we define the space
	$$W^{k, p(\cdot)}_{0}(\Omega) = \overline{C^{\infty}_{0}(\Omega)}$$ and we take the closure in $W^{k, p(\cdot)}(\Omega)$. Moreover when set $\Omega$ is bounded we have an equivalent norm on the space $W^{k, p(\cdot)}_{0}(\Omega)$ given by $\|u\| = |\nabla u|_{p(\cdot)}$.
\end{definition}
As a simple consequence of the definition of norm in $W^{1, p(\cdot)}$ we get
\begin{remark}\label{N p remark 1}
	Let  $p \in P(\Omega)$. For a sequence $\{u_{n}\}_{n \in \mathbb{N}} \subset W^{1, p(\cdot)}(\Omega)$ and a function $u \in W^{1, p(\cdot)}(\Omega)$ we have
	$$u_{n} \to u \quad \textrm{strongly in} \quad  W^{1, p(\cdot)}(\Omega)$$
	if and only if
	$$u_{n} \to u \quad \textrm{strongly in} \quad  L^{p(\cdot)}(\Omega)$$
	and
	$$\nabla u_{n} \to \nabla u \quad \textrm{strongly in} \quad  L^{p(\cdot)}(\Omega; \mathbb{R}^{N}).$$
\end{remark}

\begin{definition}\label{N p def 5}
	If function $p$ belongs to $ P(\Omega)$, then we can define the Sobolev conjugate exponent point-wise, i.e
	\begin{equation} p^{*}(z) =\label{eq_01}
	\left\{
	\begin{array}{l}
	\frac{Np(z)}{N - p(z)}
	\quad\textrm{if} \quad N > p(z) \quad  z \in \Omega\\
	\displaystyle
	+\infty
	\quad\textrm{if}\quad N \le p(z) \quad z \in \Omega.
	\end{array}
	\right.
	\end{equation}	
\end{definition}
\begin{remark}\label{N p remark 3}
	The Sobolev conjugate exponent is a generalization of the critical Sobolev exponent.
\end{remark}
\begin{remark}\label{N p remark 5}
	$p^{*}(z) \ge \frac{p^{*}(z)}{q(z)} \ge p(z)$ for all $z \in \Omega$.
\end{remark}
\begin{theorem}\label{N p theorem 7}(see \cite{Fan Zhang} Propsition 2.5)
	Let $\Omega \subset \mathbb{R}^{N}$ be bounded set with Lipschitz continuous boundary $\pO$ and functions $p, s \in C(\overline{\Omega})$ with $p(z) \ge 1$ for all $z \in \overline{\Omega}$. Let $1 \le s(z) < p^{*}(z)$ for all $z \in \overline{\Omega}$. Then the embedding $W^{1, p(\cdot)}(\Omega) \to L^{s(\cdot)}(\Omega)$ is compact. 
\end{theorem}
As a simple consequence of Theorem \ref{N p theorem 7} and the Urysohn subsequence principle we get.
\begin{remark}\label{N p remark 6}
	Under the same assumptions as in Theorem \ref{N p theorem 7} if a sequence $\{u_{n}\}_{u \in \mathbb{N}} \subset W^{1, p(\cdot)}(\Omega)$ and function $u \in  W^{1, p(\cdot)}(\Omega)$ are such that
	$$u_{n} \rightharpoonup u \quad \textrm{weakly in} \quad W^{1, p(\cdot)}(\Omega)$$ then
	$$u_{n} \to u \quad \textrm{strongly in} \quad L^{s(\cdot)}(\Omega).$$
\end{remark}
Before we formulate and prove our main result we need do some preparation.
\section{Auxillary lemmas}
\begin{definition}
	A function $a: \Omega \times \mathbb{R} \times \mathbb{R}^{N} \to \mathbb{R}^{N}$ is called a Carath\'{e}odory function
	if for all $(s, \xi) \in \mathbb{R} \times \mathbb{R}^{N}$ the function
	$$\Omega \ni z \to a(z, s, \xi) \in \mathbb{R}^{N} \quad\textrm{is measurable}$$
	and for almost all $z \in \Omega$ the function
	
	$$\mathbb{R} \times \mathbb{R}^{N} \ni (s, \xi) \to a(z, s, \xi) \in \mathbb{R}^{N} \quad\textrm{is continuous}.$$
\end{definition}

\begin{lemma}\label{A L lemma 1}
	If a function
	$$a: \Omega \times \mathbb{R} \times \mathbb{R}^{N} \to
	\mathbb{R}^{N} \quad \textrm{is Carath\'{e}odory}$$
	and functions
	$$v: \Omega \to \mathbb{R}$$
	$$w: \Omega \to \mathbb{R}^{N}$$
	are measurable, then the function
	$$\Omega \ni z \to a(z, v(z), w(z)) \in \mathbb{R}^{N}$$
	is measurable.
\end{lemma}
For a Carath\'{e}odory function $a: \Omega \times \mathbb{R} \times \mathbb{R}^{N} \to  \mathbb{R}^{N}$ and conjugate functions $p, q \in C(\overline{\Omega})$ with $1 < p(z)$ for all $z \in \overline{\Omega}$ let us consider the following assumptions: 
\begin{description}
	\item[A1]
	There exist constants $c_{0} > 0$ and a function $r_{1} \in  C(\overline{\Omega})$ such that $1 \le r_{1}(z) < \frac{p^{*}(z)}{q(z)}$ and a function $k_{0} \in L^{q(\cdot)}(\Omega)$ such that for almost all $z \in \Omega$ and for all $(s, \xi) \in \mathbb{R} \times \mathbb{R}^{N}$ we have
	
	$$|a(z, s, \xi)| \le |k_{0}(z)| + c_{0}(|s|^{r_{1}(z)} + |\xi|^{p(z) - 1})$$
\end{description}
\begin{description}
	\item[A2] 
	For almost all $z \in \Omega$ for all $s \in \mathbb{R}$ and for all $\xi, \xi^{'} \in \mathbb{R}^{N}$, $\xi \ne \xi^{'}$
	$$\sum_{i = 1}^{N}(a_{i}(z, s, \xi) - a_{i}(z, s, \xi^{'}))(\xi_{i} - \xi_{i}^{'}) > 0$$
\end{description}
\begin{description}
	\item[A3]
	There exist constants $c_{1}, c_{2} > 0$ and a function $r_{2} \in C(\overline{\Omega})$ such that $1 \le r_{2}(z) < p^{*}(z)$ and a function $k_{1} \in L^{1}(\Omega)$ such that for almost all 
	$z \in \Omega$ and for all $(s, \xi) \in \mathbb{R} \times \mathbb{R}^{N}$ we have
	$$\sum_{i = 1}^{N}a_{i}(z, s, \xi)\xi_{i} \ge c_{1}|\xi|^{p(z)} - c_{2}|s|^{r_{2}(z)} - |k_{1}(z)|$$
\end{description}
We will need two inequalities, for every $1 \le p <\infty$ we have
\begin{equation}\label{A l in 1}
\forall a, b \ge 0 \quad (a + b)^{p} \le 2^{p - 1}(a^{p} + b^{p})
\end{equation}
\begin{equation}\label{A l in 2}
\forall a, b, c \ge 0 \quad (a + b + c)^{p} \le 3^{p - 1}(a^{p} + b^{p} + c^{p}).
\end{equation}
They are a simple consequence of convexity of the function $[0,\infty) \ni t \to t^{p} \in \mathbb{R}$.
\begin{lemma}\label{A L lemma 2}
	Let function $p \in P(\Omega)$ with $1 < p_{-} \le p_{+} < +\infty$ and function $q \in P(\Omega)$ is conjugate to $p$. If the function $a: \Omega \times \mathbb{R} \times \mathbb{R}^{N} \to \mathbb{R}^{N}$ is a Carath\'{e}odory function and if it satisfies the condition A1 then the operator 
	$$A: W^{1, p(\cdot)}(\Omega) \to W^{1, p(\cdot)}(\Omega)^{*}$$
	$$\langle A(u), v\rangle = \int_{\Omega}(a(z, u(z), \nabla u(z)), \nabla v(z))_{\mathbb{R}^{N}}dz \quad \forall u, v \in  W^{1, p(\cdot)}(\Omega)$$
	
    is well defined. Moreover if the function $a$ satisfies conditions A1 - A3, then the operator $A$ is bounded, pseudomonotone and coercive.
\end{lemma}
\begin{proof}
Using Lemma \ref{A L lemma 1} we easily see that for every $u, v \in W^{1, p(\cdot)}(\Omega)$ the function
$$\Omega \ni z \to (a(z, u(z), \nabla u(z)), \nabla v(z))_{\mathbb{R}^{N}} \in \mathbb{R}$$ is measurable.
Now let us notice that for every $u, v \in W^{1,p(\cdot)}(\Omega)$ the integral
$$ \int_{\Omega}(a(z, u(z), \nabla u(z)), \nabla v(z))_{\mathbb{R}^{N}}dz$$
is finite. Indeed using Schwarz inequality, condition A1 and H\"{o}lder's inequality we have
\begin{eqnarray}\label{A L in}
& & 
\int_{\Omega}|(a(z, u(z), \nabla u(z)), \nabla v(z))_{\mathbb{R}^{N}}|dz\cr
& &
\le \int_{\Omega}|(a(z, u(z), \nabla u(z))|\cdot|\nabla v(z)|dz\cr
& &
\le \int_{\Omega}  (|k_{0}(z)| + c_{0}(|u(z)|^{r_{1}(z)} + |\nabla u(z)|^{p(z) - 1})) \cdot |\nabla v(z)|dz\cr
& &
= \int_{\Omega} |k_{0}(z)|\cdot|\nabla v(z)|dz + c_{0}\int_{\Omega}|u(z)|^{r_{1}(z)}\cdot |\nabla v(z)|dz\cr
& &
+ c_{0}\int_{\Omega}|\nabla u(z)|^{p(z) - 1}\cdot |\nabla v(z)|dz\cr
& &
\le 2\|k_{0}\|_{q(\cdot)}\|\nabla v\|_{p(\cdot)} + 2c_{0}\|u^{r_{1}}\|_{q(\cdot)}
\|\nabla v\|_{p(\cdot)} + 2c_{0}\||\nabla u|^{p(\cdot) - 1}\|_{q(\cdot)}\|\nabla v\|_{p(\cdot)}\cr
\end{eqnarray}

Next let us observe that for every $u \in W^{1, p(\cdot)}(\Omega)$ the map
$$W^{1, p(\cdot)}(\Omega) \ni v \to \int_{\Omega}(a(z, u(z), \nabla u(z)), \nabla v(z))_{\mathbb{R}^{N}}dz \in \mathbb{R}$$
is obviously linear. Using the inequality (\ref{A L in}) for every $v \in W^{1, p(\cdot)}(\Omega)$ we have
$$\quad \langle A(u), v \rangle \le 2(\|k_{0}\|_{q(\cdot)} + c_{0}\|x^{r_{1}}\|_{q(\cdot)} + \||\nabla u|^{p(\cdot) - 1}\|_{q(\cdot)})\cdot
\|\nabla v\|_{p(\cdot)}$$
$$\quad \le2(\|k_{0}\|_{q(\cdot)} + c_{0}\|u^{r_{1}}\|_{q(\cdot)} + c_{0}\||\nabla u|^{p(\cdot) - 1}\|_{q(\cdot)})\cdot
\|v\|_{W^{1, p(\cdot)}(\Omega)}$$
and we see the this map is also continuous. For the second part let us observe that we can repeat proofs from Lemma 2.31, 2.32 and 2.35 in \cite{Roubicek}.
\end{proof}
In next Lemmas \ref{A L lemma 3} - \ref{A L lemma 5}  we assume that $\Omega$ is bounded with Lipschitz continuous boundary $\pO$. We also assume that $p \in C(\overline{\Omega})$ with $1 < p(z)$ for all $z \in \overline{\Omega}$ and  one of the following conditions is satisfied:
$$p(z) < N \quad \forall z \in  \overline{\Omega}$$
$$p(z) = N \quad \forall z \in \overline{\Omega}$$
$$p(z) > N  \quad \forall z \in \overline{\Omega}.$$
We assume the 
$a: \Omega \times \mathbb{R} \times \mathbb{R}^{N} \to \mathbb{R}^{N}$ is Carath\'{e}odory function. Let us also fix a sequence $\{u_{n}\}_{n \in \mathbb{N}} \subset W^{1, p(\cdot)}(\Omega)$ and a function $u \in  W^{1, p(\cdot)}(\Omega)$ such that
$$u_{n}  \rightharpoonup u \quad \textrm{weakly in} \quad W^{1, p(\cdot)}(\Omega).$$
\begin{lemma}\label{A L lemma 3}
	 If the function $a$ satisfies condition A1, then from the sequence $\{u_{n}\}_{n \in \mathbb{N}}$
 we can choose a subsequence $\{u_{n_{k}}\}_{k \in \mathbb{N}}$ such that
$$\lim_{k \to +\infty}\int_{\Omega}|(a(z, u_{n_{k}}(z), \nabla u(z)) -
a(z, u(z), \nabla u(z)),
\nabla u_{n_{k}}(z) - \nabla u(z))_{\mathbb{R}^{N}}|dz = 0.$$	
\end{lemma}
\begin{proof}
First we show that for every $n \in \mathbb{N}$ the function
$$a(\cdot, u_{n}(\cdot), \nabla u(\cdot)) -
a(\cdot, u(\cdot), \nabla u(\cdot)) \in L^{q(\cdot)}(\Omega).$$
Indeed using inequalities (\ref{A l in 1}), (\ref{A l in 2}) and condition A1 we have
\begin{align*}
& 
\int_{\Omega}|a(z, u_{n}(z), \nabla u(z)) -
a(z, u(z), \nabla u(z))|^{q(z)}dz\\
& 
\quad \le \int_{\Omega}2^{q(z) - 1}\left(|a(z, u_{n}(z), \nabla u(z))|^{q(z)}
+ |a(z, u(z), \nabla u(z))|^{q(z)}\right)dz\\
& 
\quad \le \int_{\Omega}2^{q(z) - 1}\left(|k_{0}(z)| + c_{0}(|u_{n}(z)|^{r_{1}(z)} + |\nabla u(z)|^{p(z) - 1}\right)^{q(z)}dz\\
& 
\qquad + \int_{\Omega}2^{q(z) - 1}\left(|k_{0}(z)| + c_{0}(|u(z)|^{r_{1}(z)} + |\nabla u(z)|^{p(z) - 1}\right)^{q(z)}dz\\
& 
\quad \le \int_{\Omega}2^{q(z) - 1}3^{q(z) - 1}\left(|k_{0}(z)|^{q(z)} + c_{0}^{q(z)}(|u_{n}(z)|^{r_{1}(z)q(z)} + |\nabla u(z)|^{p(z)}\right)dz\\
& 
\qquad + \int_{\Omega}2^{q(z) - 1}3^{q(z) - 1}\left(|k_{0}(z)|^{q(z)} + c_{0}^{q(z)}(|u(z)|^{r_{1}(z)q(z)} + |\nabla u(z)|^{p(z)}\right)dz.\\
\end{align*}	
For every $n \in \mathbb{N}$ using Schwarz inequality and H\"{o}lder's inequality we obtain
\begin{eqnarray}
& &
\int_{\Omega}|(a(z, u_{n}(z), \nabla u(z)) -
a(z, u(z), \nabla u(z)),
\nabla u_{n}(z) - \nabla u(z))_{\mathbb{R}^{N}}|dz\cr
& &
\quad \le \int_{\Omega}|(a(z, u_{n}(z), \nabla u(z)) -
a(z, u(z), \nabla u(z))|\cdot
|\nabla u_{n}(z) - \nabla u(z)|dz\cr
& &
\quad \le 2\|(a(z, u_{n}(z), \nabla u(z)) -
a(z, u(z), \nabla u(z))\|_{q(\cdot)}\cdot
\|\nabla u_{n} - \nabla u\|_{p(\cdot)}\cr
\end{eqnarray}

Let us observe that the sequence $\{\|\nabla u_{n} - \nabla u\|_{p(\cdot)}\}_{n \in \mathbb{N}}$ is bounded.

Using Lemma \ref{N p lemma 1} it is enough to find such subsequence $\{u_{n_{k}}\}_{k \in \mathbb{N}}$ that
$$\lim_{k \to +\infty} \int_{\Omega}|a(z, u_{n_{k}}(z), \nabla u(z)) -
a(z, u(z), \nabla u(z))|^{q(z)}dz = 0.$$
Using Remark \ref{N p remark 6} we know that
$$u_{n} \to u \quad \textrm{strongly in} \quad L^{r_{1}(\cdot)q(\cdot)}(\Omega).$$ 
So we can choose a subsequence 
$\{u_{n_{k}}\}_{k \in \mathbb{N}}$ and find a function
$g \in L^{r_{1}(\cdot)q(\cdot)}(\Omega)$ such that
\begin{equation}\label{A l eq1}
u_{n_{k}}(z) \to u(z) \quad \textrm{for almost all} \quad z \in \Omega
\end{equation} 
and
\begin{equation}\label{A l eq2}
|u_{n_{k}}(z)| \le |g(z)| \quad \textrm{for almost all} \quad z \in \Omega \quad \textrm{and} \quad \forall k \in \mathbb{N}.
\end{equation}
We check that the assumptions of the Lebesgue convergence theorem are satisfied for this subsequence. Using (\ref{A l eq1}) and the fact that function $a$ is a Carath\'{e}odory function we have for almost all $z \in \Omega$
$$\lim_{k \to +\infty} |a(z, u_{n_{k}}(z), \nabla u(z)) -
a(z, u(z), \nabla u(z))|^{q(z)} = 0.$$
Using inequalities (\ref{A l in 1}), (\ref{A l in 2}), condition A1 and (\ref{A l eq2}) for every $k \in \mathbb{N}$ and for almost all $z \in \Omega$ we get
\begin{align*}
&
|a(z, u_{n_{k}}(z), \nabla u(z)) -
a(z, u(z), \nabla u(z))|^{q(z)}\\
&\ \le 2^{q(z) - 1} 3^{q(z) - 1}\cdot\\
&\ \cdot\left(2|k_{0}(z)|^{q(z)} + c_{0}^{q(z)}(|u_{n_{k}}(z)|^{r_{1}(z)q(z)} + c_{0}^{q(z)}(|u(z)|^{r_{1}(z)q(z)} + 2c_{0}^{q(z)}|\nabla u(z)|^{p(z)}\right)\\
&\  \le 2^{q(z) - 1} 3^{q(z) - 1}\cdot\\
&\ \cdot\left(2|k_{0}(z)|^{q(z)} + c_{0}^{q(z)}(|g(z)|^{r_{1}(z)q(z)} + c_{0}^{q(z)}(|u(z)|^{r_{1}(z)q(z)} + 2c_{0}^{q(z)}|\nabla u(z)|^{p(z)}\right).
\end{align*}
So using the Lebesgue convergence theorem we finally get
$$\lim_{k \to +\infty} \int_{\Omega}|a(z, u_{n_{k}}(z), \nabla u(z)) -
a(z, u(z), \nabla u(z))|^{q(z)}dz = 0.
$$
\end{proof}
Let us denote for all $z \in \Omega$ and for all $n \in \mathbb{N}$
\begin{equation}
\theta^{1}_{n}(z) =  (a(z, u_{n}(z), \nabla u_{n}(z)) - a(z, u_{n}(z), \nabla u(z)), \nabla u_{n}(z) - \nabla u(z))_{\mathbb{R}^{N}}
\end{equation}
and
\begin{equation}
\theta^{2}_{n}(z) =  (a(z, u_{n}(z), \nabla u(z)) - a(z, u(z), \nabla u(z)), \nabla u_{n}(z) - \nabla u(z))_{\mathbb{R}^{N}}.
\end{equation}
We will use this notation in Lemma \ref{A L lemma 4} and Lemma \ref{A L lemma 5}.
 \begin{lemma}\label{A L lemma 4}
 	 If the function $a$ satisfies conditions A1 and A2, then from the sequence $\{u_{n}\}_{n \in \mathbb{N}} \subset W^{1, p(\cdot)}(\Omega)$ such that
 	$$\limsup_{n \to +\infty}\langle A(x_{n}), x_{n} - x\rangle \le 0$$
 	 we can choose a subsequence $\{u_{n_{k}}\}_{k \in \mathbb{N}}$ such that
 	$$\lim_{k \to +\infty}\langle A(u_{n_{k}}) - A(u), u_{n_{k}} - u\rangle = 0.$$
 \end{lemma}
\begin{proof}
Because $u_{n} \rightharpoonup x$ converges weakly in $W^{1, p(\cdot)}(\Omega)$ so $\langle A(u),x_{u}  - u\rangle \to 0$ and in a consequence we have
\begin{equation}\label{L 4 eq 1}
\limsup_{n \to +\infty}\langle A(u_{n}) - A(u),u_{n} - u\rangle \le 0.
\end{equation}
Next using Lemma \ref{A L lemma 3}  we can find such subsequence $\{u_{n_{k}}\}_{k \in \mathbb{N}}$ that
\begin{equation}\label{L 4 eq 2}
\int_{\Omega}\theta^{2}_{n_{k}}(z)dz \to 0.
\end{equation}
Using condition A2 we see that for every $n \in \mathbb{N}$
\begin{equation}\label{L 4 eq 3}
\int_{\Omega}\theta^{1}_{n}(z)dz \ge 0.
\end{equation} 
On the other hand for every $n \in \mathbb{N}$ we can write
\begin{equation}\label{L 4 eq 4}
\langle A(u_{n}) - A(u), u_{n} - u\rangle = \int_{\Omega}\theta^{1}_{n}(z)dz + \int_{\Omega}\theta^{2}_{n}(z)dz.
\end{equation}

Now using (\ref{L 4 eq 2}), (\ref{L 4 eq 3}) and (\ref{L 4 eq 4}) we can conclude that 
\begin{equation}\label{L 4 eq 5}
\liminf_{k \to +\infty}\langle A(u_{n_{k}}) - A(u),u_{n_{k}} - u\rangle \ge 0
\end{equation}

and finally combining (\ref{L 4 eq 1}) and  (\ref{L 4 eq 5}) we get
$$\lim_{k \to +\infty}\langle A(u_{n_{k}}) - A(u),u_{n_{k}} - u\rangle = 0.$$
\end{proof}
Let us denote for every $n \in \mathbb{N}$ and for every $z \in \Omega$
$$\xi_{n}(z) = (a(z, u_{n}(z), \nabla u_{n}(z)) - a(z, u(z), \nabla u(z)), \nabla u_{n}(z) - \nabla u(z))_{\mathbb{R}^{N}}.$$
\begin{lemma}\label{A L lemma 5}
		 If the function $a$ satisfies conditions A1 and A2, then we can find a subsequence $\{\xi_{n_{k}}\}_{k \in \mathbb{N}}$ and a function $l \in L^{1}(\Omega)_{+}$ such that
	$$\xi_{n_{k}}(z) \to 0 \quad \textrm{for almost all}\quad z \in \Omega$$and
	
	$$|\xi_{n_{k}}(z)| \le l(z) \quad  \textrm{for almost all}\quad z \in \Omega \quad \textrm{and}\quad \forall k \in \mathbb{N}.$$ 
\end{lemma}
\begin{proof}
Using Lemma \ref{A L lemma 3} we know that we can find a subsequence $\{x_{n_{k}}\}_{k \in \mathbb{N}}$ such that 
\begin{equation}\label{Lemma5 eq 1}
\theta^{2}_{n_{k}} \to 0 \quad\textrm{strongly in}\quad L^{1}(\Omega).
\end{equation}

Using condition A2 we see that for every $n \in \mathbb{N}$ and for almost all $z \in \Omega$ we have
\begin{equation}\label{Lemma5 eq 2}
\theta^{1}_{n}(z) \ge 0.
\end{equation}

We know from Lemma \ref{A L lemma 4} that
\begin{equation}\label{Lemma5 eq 3}
\int_{\Omega}\xi_{n_{k}}(z)dz \to 0.
\end{equation}

Because for all  $z \in \Omega$ and for all $n \in \mathbb{N}$ we have
\begin{equation}\label{Lemma5 eq 4}
\xi_{n}(z) = \theta^{1}_{n}(z) + \theta^{2}_{n}(z)
\end{equation}

so combining (\ref{Lemma5 eq 1}), (\ref{Lemma5 eq 3}) and (\ref{Lemma5 eq 4}) we can conclude that
\begin{equation}\label{Lemma5 eq 5}
\theta^{1}_{n_{k}} \to 0 \quad\textrm{strongly in}\quad L^{1}(\Omega)
\end{equation}

and hence using (\ref{Lemma5 eq 2}) and (\ref{Lemma5 eq 5}) we have

$$\xi_{n_{k}} \to 0 \quad\textrm{strongly in}\quad L^{1}(\Omega).$$

At the end we can extract a subsubsequence $\{\xi_{n_{k_{j}}}\}_{j \in \mathbb{N}}$ and find a function $l \in L^{1}(\Omega)_{+}$ such that
$$\xi_{n_{k_{j}}}(z) \to 0 \quad\textrm{for almost all}\quad z \in \Omega$$
and
$$|\xi_{n_{k_{j}}}(z)| \le l(z) \quad\textrm{for almost all}\quad z \in \Omega \quad \textrm{and} \quad \forall j \in \mathbb{N}.$$
\end{proof}	
Now we ready to prove the main result.
\section{Main result}
\begin{theorem}\label{Main theorem}
Let $\Omega$ be a bounded set with Lipschitz continuous boundary $\pO$. Let us assume that function $p \in C(\overline{\Omega})$ with $1 < p(z)$ for all $z \in \overline{\Omega}$ and  let $a: \Omega \times \mathbb{R} \times \mathbb{R}^{N} \to \mathbb{R}^{N}$ be a Carath\'{e}odory function such that it satisfies conditions A1 - A3. Let us also assume that one of the following conditions hold
$$p(z) < N \quad \forall z \in \overline{\Omega}$$
$$p(z) = N \quad \forall z \in \overline{\Omega}$$
$$p(z) > N  \quad \forall z \in \overline{\Omega}.$$
 Then the operator
$$A: W^{1, p(\cdot)}(\Omega) \to W^{1, p(\cdot)}(\Omega)^{*}$$
$$\langle A(u), v\rangle = \int_{\Omega}(a(z, u(z), \nabla u(z)), \nabla v(z))_{\mathbb{R}^{N}}dz \quad \forall u, v \in W^{1, p(\cdot)}(\Omega)$$
has S+ property. 
\end{theorem}
\begin{proof}
Let us take a sequence $\{u_{n}\}_{n \in \mathbb{N}} \subset W^{1, p(\cdot)}(\Omega)$ such that
$$u_{n} \rightharpoonup u \quad \textrm{weakly in}\quad  W^{1, p(\cdot)}(\Omega)$$
and
$$\limsup_{n \to +\infty}\langle A(u_{n}), u_{n} - u\rangle \le 0.$$
\newline 
First let us observe that by Remark \ref{N p remark 6} the convergence holds
$$u_{n} \to u \quad \textrm{strongly in} \quad L^{p(\cdot)}(\Omega).$$
Now by Remark \ref{N p remark 1} it is enough to show that
\begin{equation}
\nabla u_{n} \to \nabla u \quad \textrm{strongly in} \quad L^{p(\cdot)}(\Omega; \mathbb{R}^{N}).
\end{equation}
Let us observe that using the Urysohn subsequence principle it is enough to find some subsequence $\{u_{n_{k}}\}_{k \in \mathbb{N}}$ such that
$$\nabla u_{n_{k}} \to \nabla u \quad \textrm{strongly in} \quad L^{p(\cdot)}(\Omega; \mathbb{R}^{N}).$$
Now we show how to find such subsequence. Using Remark \ref{N p remark 6} and Lemma \ref{A L lemma 5} we can choose a subsequence
$\{u_{n_{k}}\}_{k \in \mathbb{N}}$ and find functions $g_{1} \in L^{r_{1}(\cdot)}(\Omega)$, $g_{2} \in  L^{r_{2}(\cdot)}(\Omega)$ and $l \in L^{1}(\Omega)_{+}$ such that for almost all $z \in \Omega$ and for every $k \in \mathbb{N}$:
\begin{equation}\label{M r eq 1}
u_{n_{k}}(z) \to u(z)
\end{equation}
\begin{equation}\label{M r eq 2}
|u_{n_{k}}(z)| \le g_{1}(z) 
\end{equation}
\begin{equation}\label{M r eq 3}
|u_{n_{k}}(z)| \le g_{2}(z) 
\end{equation}
\begin{equation}\label{M r eq 4}
(a(z, u_{n_{k}}(z), \nabla u_{n_{k}}(z)) - a(z, u(z), \nabla u(z)), \nabla u_{n_{k}}(z) - \nabla u(z))_{\mathbb{R}^{N}} \to 0
\end{equation}
\begin{equation}\label{M r eq 5}
l(z) \ge (a(z, u_{n_{k}}(z), \nabla u_{n_{k}}(z)) - a(z, u(z), \nabla u(z)), \nabla u_{n_{k}}(z) - \nabla u(z))_{\mathbb{R}^{N}}.  
\end{equation}
We set $\zeta_{k} = u_{n_{k}}$ for all $k \in \mathbb{N}$. We want to show that
$$\nabla \zeta_{k} \to \nabla u \quad \textrm{strongly in} \quad L^{p(\cdot)}(\Omega;\mathbb{R}^{N}).$$
To prove this we use the Vitali convergence theorem. In the first step we show that
$$\nabla \zeta_{k}(z) \to \nabla u(z) \quad \textrm{for almost all} \quad z \in \Omega.$$ Using conditions A1 and A3 we have for every $k \in \mathbb{N}$ and for almost all $z \in \Omega$
\begin{eqnarray}\label{M r eq 6}
& &
l(z) \ge (a(z, \zeta_{k}(z), \nabla \zeta_{k}(z)) - a(z, u(z), \nabla u(z)), \nabla \zeta_{k}(z) - \nabla u(z))_{\mathbb{R}^{N}}\cr
& &\quad\ \ 
\ge c_{1}|\nabla \zeta_{k}(z)|^{p(z)} -c_{2}|\zeta_{k}(z)|^{r_{2}(z)} - |k_{1}(z)| \cr
& &\qquad\ 
+c_{1}|\nabla u(z)|^{p(z)}
-c_{2}|u(z)|^{r_{2}(z)} - |k_{1}(z)|\cr
& &
\qquad \ - (|k_{0}(z)| + c_{0}(|\zeta_{k}(z)|^{r_{1}(z)} + |\nabla \zeta_{k}(z)|^{p(z) - 1}))|\nabla u(z)|\cr 
& &
\qquad \ - (|k_{0}(z)| + c_{0}(|u(z)|^{r_{1}(z)} + |\nabla u(z)|^{p(z) - 1}))|\nabla \zeta_{k}(z)|\cr
& &\quad\ \ 
\ge c_{1}|\nabla \zeta_{k}(z)|^{p(z)} -c_{2}|g_{2}(z)|^{r_{2}(z)} - |k_{1}(z)|\cr
& &\qquad\ 
 + c_{1}|\nabla u(z)|^{p(z)}
-c_{2}|u(z)|^{r_{2}(z)} - |k_{1}(z)|\cr
& &
\qquad \ - (|k_{0}(z)| + c_{0}(|g_{1}(z)|^{r_{1}(z)} + |\nabla \zeta_{k}(z)|^{p(z) - 1}))|\nabla u(z)|\cr
& &
\qquad \ - (|k_{0}(z)| + c_{0}(|u(z)|^{r_{1}(z)} + |\nabla u(z)|^{p(z) - 1}))|\nabla \zeta_{k}(z)|.\cr
\end{eqnarray}

Let us denote the set
$$\Omega_{0} = \{z \in \Omega \quad|\quad\{\nabla \zeta_{k}(z)\}_{k \in \mathbb{N}} \quad \textrm{is bounded in}\quad \mathbb{R}^{N}\}.$$
From inequality (\ref{M r eq 6}) we can conclude that the set $\Omega_{0}$ is of full measure in $\Omega$. For all $z \in \Omega_{0}$
let us fix an arbitrary subsequence $\{\nabla \zeta_{k_{j_{z}}}(z)\}_{j_{z} \in \mathbb{N}}$ depending on $z$. Now we can choose some subsubsequence $\{\nabla \zeta_{k_{j_{w_{z}}}}(z)\}_{w_{z} \in \mathbb{N}}$ such that it converges in $\mathbb{R}^{N}$. Let us define the function $\beta: \Omega \to \mathbb{R}^{N}$

$$\beta(z) = \lim_{w_{z} \to +\infty}\nabla \zeta_{k_{j_{w_{z}}}}(z) \quad \textrm{for almost all} \quad z \in \Omega.$$

From (\ref{M r eq 4}) we know that
\begin{equation}\label{M r eq 7}
(a(z, \zeta_{k_{j_{w_{z}}}}(z), \nabla \zeta_{k_{j_{w_{z}}}}(z)) - a(z, u(z),\nabla u(z)),\nabla \zeta_{k_{j_{w_{z}}}}(z) - \nabla u(z))_{\mathbb{R}^{N}} \to 0.
\end{equation} 
Now passing to the limit in (\ref{M r eq 7}), using (\ref{M r eq 1}) and the fact that function $a$ is Carath\'{e}odory we obtain for almost all $z \in \Omega$
\begin{equation}\label{M r eq 8}
(a(z, u(z), \beta (z)) - a(z, u(z),\nabla u(z)),\beta (z) - \nabla u(z))_{\mathbb{R}^{N}} = 0.
\end{equation}

Now combining (\ref{M r eq 8}) and condition A2 we get
$$\nabla u(z) = \beta (z) \quad\textrm{for almost all}\quad z \in \Omega.$$
Finally using the Urysohn subsequence principle this proves that
$$\nabla \zeta_{k}(z) \to \nabla  u(z) \quad\textrm{for almost all}\quad z \in \Omega.$$
In the second step we prove the uniform integrability.
We will need the following fact from measure theory.
\begin{remark}\label{M R remark 1}
	 Let function $k \in L^{1}(\Omega)$ then for every $\epsilon > 0$ we can find $\delta(\epsilon) > 0$ such that for every measurable set $C \subset \Omega$ such that $\mu(C) < \delta(\epsilon)$ 
	$$\int_{C}|k(z)|dz < \epsilon.$$
\end{remark}
Let $A > 0$ be such that $\|\nabla \zeta_{k}\|_{p(\cdot)} < A$ for all $k \in \mathbb{N}$. Using the inequality (\ref{M r eq 6}) and H\"{o}lder's inequality for every measurable set $C \subset \Omega$ we have
\begin{eqnarray}\label{M R in}
& &
c_{1}\int_{C}|\nabla \zeta_{k}(z)|^{p}dz\cr
& &
\le \int_{C}l(z)dz +
2\|k_{0}\|_{q(\cdot)}\|\nabla u\|_{p(\cdot)}
+ 2c_{0} \|g_{1}^{r_{1}}\|_{q(\cdot)}\|\nabla u\|_{p(\cdot)}
+ 2c_{0}A\|\nabla u\|_{p(\cdot)}\cr
& &
+ 2A\|k\|_{q(\cdot)}
+ 2c_{0}A\|x^{r_{1}}\|_{q(\cdot)}
+ 2c_{0}A\|\nabla x\|_{p(\cdot)}\cr
& &
 + c_{1}\int_{C}|g_{1}(z)|^{r_{2}(z)}dz +  c_{1}\int_{C}|u(z)|^{r_{2}(z)}dz  + 2\int_{C}|k_{1}(z)|dz.\cr
\end{eqnarray}
Now combining (\ref{M R in}), \ref{M R remark 1} and \ref{N p lemma 1} we finish the proof.
\end{proof}
\begin{remark}
	Under the same assumptions as in \ref{Main theorem} reapeting the proof of Theorem \ref{Main theorem} we see that  S+ property also holds for operator $A: W_{0}^{1, p(\cdot)}(\Omega) \to W^{1, p(\cdot)}_{0}(\Omega)^{*}$ defined by
	$$\langle A(u), v \rangle = \int_{\Omega}(a(z, u(z),\nabla u(z), \nabla v(z))_{\mathbb{R}^{N}}dz \quad \forall u, v \in W^{1, p(\cdot)}_{0}(\Omega).$$
	Let us observe that in this case no regularity condition on boundary $\pO$ is needed.
\end{remark}

\section{Examples of operators satisfying S+ property}

\begin{example}(see \cite{M M Papa} Example 2.2.7 and Proposition 2.72)
	Let $\Omega \subset \mathbb{R}^{N}$ be a bounded domain and $1 < p < +\infty$. Then the operator $A: W_{0}^{1, p}(\Omega) \to W^{1, p}_{0}(\Omega)^{*}$ defined by
	$$\langle A(u), v \rangle = \int_{\Omega}|\nabla u(z)|^{p - 2}(\nabla u(z), \nabla v(z))_{\mathbb{R}^{N}}dz \quad \forall u, v \in W^{1, p}_{0}(\Omega)$$ has S+ property.
\end{example}
\begin{example}(see \cite{M M Papa} Proposition 2.72)
	Let $\Omega \subset \mathbb{R}^{N}$ be a bounded domain with Lipschitz continuous boundary and $1 < p < +\infty$. Then the operator $A: W^{1, p}(\Omega) \to W^{1, p}(\Omega)^{*}$ defined by
	$$\langle A(u), v \rangle = \int_{\Omega}|\nabla u(z)|^{p - 2}(\nabla u(z), \nabla v(z))_{\mathbb{R}^{N}}dz \quad \forall u, v \in W^{1, p}(\Omega)$$ has S+ property. 
\end{example}
\begin{example}(see \cite{M M Papa} Proposition 12.14 and \cite{Gas Papa 1} Proposition 3.1)
	Let $\Omega \subset \mathbb{R}^{N}$ be a bounded set with Lipschitz continuous boundary. Let $a: \Omega \times \mathbb{R}^{N} \to \mathbb{R}^{N}$ be a Carath\'{e}odory function. Moreover the function $a$ satisfies
	\item[(i)] $a$ is strictly monotone on the last variable
	\item[(ii)] $(a(z, \xi), \xi)_{\mathbb{R}^{N}} \ge c_{0}|\xi|^{p}$ for almost all $z \in \Omega$ and for all $\xi \in \mathbb{R}^{N}$
	\item[(iii)] there exist function $a_{0} \in L^{\infty}(\Omega)_{+}$ and a constant $c > 0$ such that
	$$|a(z, \xi)| \le a_{0}(z) + c|\xi|^{p - 1}.$$
	Then the operator $A: W^{1, p}(\Omega) \to W^{1, p}(\Omega)^{*}$ defined by
	$$\langle A(u), v\rangle = \int_{\Omega}(a(z, \nabla u(z)), \nabla v(z))_{\mathbb{R}^{N}}dz \quad \forall u, v \in W^{1, p}(\Omega)$$ has S+ property.
\end{example}

\begin{example}(see \cite{Fan Zhang} Theorem 3.1)
	Let $\Omega \subset \mathbb{R}^{N}$ be a bounded domain and $p \in C(\overline{\Omega})$ with $1 < p(x)$ for every $x \in \overline{\Omega}$. Then the operator $L: W^{1, p(\cdot)}_{0}(\Omega) \to  W^{1, p(\cdot)}_{0}(\Omega)^{*}$ defined by
	$$\langle A(x), y \rangle = \int_{\Omega}|\nabla x(z)|^{p(z) - 2}(\nabla x(z), \nabla y(z))_{\mathbb{R}^{N}}dz \quad \forall u, v \in W^{1, p(\cdot)}(\Omega)$$  
	has S+ property.
\end{example}
\begin{example}(see \cite{Ali Gha Ke} Theorem 3.1)
	Let $\Omega \subset \mathbb{R}^{N}$ ($N \ge 2$) be a bounded smooth domain and $a \in L^{\infty}(\Omega)$ with $p, q \in C(\overline{\Omega})$ and $1 < p(z)$ for all $z \in \overline{\Omega}$. Then we can consider functional $\Psi: W^{1, p(\cdot)}(\Omega) \to \mathbb{R}$ defined by
	$$\Psi(u) = \int_{\Omega}\frac{1}{p(z)}(|\nabla u(z)|^{p(z)} + a(z)|u(z)|^{p(z)})dz \quad \forall u \in W^{1, p(\cdot)}(\Omega).$$
	Then the Fr\'{e}chet derivative $\Psi'$ has S+ property.
\end{example}

\begin{example}(see \cite{Mih Rad} Lemma 3.11)
	Let $\Omega \subset \mathbb{R}^{N}$ be a bounded domain with smooth boundary and function $p \in C(\overline{\Omega})$ and $1 < p(z)$ for every $z \in \overline{\Omega}$. Let function $A: \overline{\Omega} \times \mathbb{R}^{N} \to \mathbb{R}^{N}$ be continuous with continuous derivative with respect to $\xi$ i.e $a(z, \xi) = \nabla_{\xi}A(z, \xi)$. Suppose that functions $a$ and $A$ satisfy the following hypotheses:
	\item[(i)] There exists a positive constant $c_{1}$ such that
	$$|a(z, \xi)| \le c_{1}(1 + |\xi|^{p(z) - 1}),$$
	for all $z \in \overline{\Omega}$ and all $\xi \in \mathbb{R}^{N}$
	\item[(ii)] There exists $k > 0$ such that
	$$A(z, \frac{\xi + \psi}{2}) \le \frac{1}{2}A(z, \xi) + \frac{1}{2}A(z, \psi) - k|\xi - \psi|^{p(z)},$$
	for all $z \in \overline{\Omega}$ and $\xi, \psi \in \mathbb{R}^{N}$,
	\item[(iii)] The following inequalities hold true
	$$|\xi|^{p(z)} \le a(x, \xi)\cdot \xi$$
	for all $x \in \overline{\Omega}$ and $\xi, \psi \in \mathbb{R}^{N}$.
	If hypotheses (i) - (iii) hold then the operator $L: W^{1, p(\cdot)}_{0}(\Omega) \to W^{1, p(\cdot)}_{0}(\Omega)^{*}$ defined by
	$$\langle L(u), v \rangle = \int_{\Omega}(a(u, \nabla u), \nabla v)dz  \quad \forall u, v \in W^{1, p(\cdot)}_{0}(\Omega)$$
	has S+ type property. 
\end{example}

\begin{example}\label{Example multi}(see \cite{Gas Papa} Proposition 4.4.1)
	Let $\Omega \subset \mathbb{R}^{N}$ be a bounded domain with $C^{1}$- boundary $\pO$. Let $a: \Omega \times \mathbb{R} \times \mathbb{R}^{N} \to P_{kc}(\mathbb{R}^{N})$ be a multifunction such that the following assumptions are satisfied:
	\item[(i)] $a$ is graph measurable
	\item[(ii)] for almost all $z \in \Omega$ and all $\zeta \in \mathbb{R}$, the function
	$$\mathbb{R}^{N} \ni \xi \to a(z, \zeta, \xi) \in P_{kc}(\mathbb{R}^{N})$$
	is strictly monotone
	\item[(iii)] for almost all $z \in \Omega$, the function
	$$\mathbb{R} \times \mathbb{R}^{N} \to a(z, \zeta, \xi) \in P_{kc}(\Omega)$$
	has closed graph and for almost all $z \in \Omega$ and all $\xi \in \mathbb{R}^{N}$, the function
	$$\mathbb{R} \ni \zeta \to a(z, \zeta, \xi) \in P_{kc}(\mathbb{R}^{N})$$
	is lower semicontinuous
	\item[(iv)] for almost all $z \in \Omega$, all $\zeta \in \mathbb{R}$, all $\xi \in \mathbb{R}^{N}$ and all $v \in a(z, \zeta, \xi)$, we have that
	$$|v| \le b_{1}(z) + c_{1}(|\zeta|^{p - 1} + |\xi|^{p - 1}),$$
	with $b_{1} \in L^{p^{'}}(\Omega)_{+}, c_{1} > 0, p \in (1, +\infty), \frac{1}{p} + \frac{1}{p^{'}} = 1$
	\item[(v)] for almost all $z \in \Omega$, all $\zeta \in \mathbb{R}$, all $\xi \in \mathbb{R}^{N}$ and all $v \in a(z, \zeta, \xi)$ we have that
	$$(v, \xi)_{\mathbb{R}^{N}} \ge \eta_{1}|\xi|^{p} - \eta_{2}$$
	with $\eta_{1}, \eta_{2} > 0.$
\end{example}

	 Let us consider the operator $V: W^{1, p}_{0} \to P_{wkc}(W^{-1, p'}(\Omega))$. For all $x \in  W^{1, p}_{0}(\Omega)$ we define
	$$V(x) = \{-\divv v: v \in S^{p'}_{a(,x(),\nabla x)}\}.$$
If conditions (i)-(v) are satisfied then the operator $V$ has S+ property.

As a consequence from Example \ref{Example multi} we get S+ property for divergence operator $-\divv a(z, u(z), \nabla u(z))$ with $p$-growth.

We show an application of S+ property.
\section{Application of S+ property}
\subsection{Strong convergence of the Galerkin method}
Let us assume $\Omega \subset \mathbb{R}^{N}$ is a bounded set with Lipschitz continuous boundary $\pO$. We also assume that function $p \in C(\overline{\Omega})$ with $1 < p(z)$ for all $z \in \overline{\Omega}$  and  one of the following conditions is satisfied:
$$p(z) < N \quad \forall z \in \overline{\Omega}$$
$$p(z) = N \quad \forall z \in \overline{\Omega}$$
$$p(z) > N  \quad \forall z \in \overline{\Omega}.$$
We assume that function 
$a: \Omega \times \mathbb{R} \times \mathbb{R}^{N} \to \mathbb{R}^{N}$ is a Carath\'{e}odory function and we consider the following assumptions: 
\begin{description}
	\item[A1]
	There exist constants $c_{0} > 0$ and a function $r_{1} \in C(\overline{\Omega})$ such that
	$1 \le r_{1}(z) < \frac{p^{*}(z)}{q(z)}$ and a function $k_{0} \in L^{q(\cdot)}(\Omega)$ such that for almost all $z \in \Omega$ and for all $(s, \xi) \in \mathbb{R} \times \mathbb{R}^{N}$ we have
	
	$$|a(z, s, \xi)| \le |k_{0}(z)| + c_{0}(|s|^{r_{1}(z)} + |\xi|^{p(z) - 1})$$
\end{description}
\begin{description}
	\item[A2] 
	For almost all $z \in \Omega$ for all $s \in \mathbb{R}$ and for all $\xi, \xi^{'} \in \mathbb{R}^{N}$, $\xi \ne \xi^{'}$
	$$\sum_{i = 1}^{N}(a_{i}(z, s, \xi) - a_{i}(z, s, \xi^{'}))(\xi_{i} - \xi_{i}^{'}) > 0$$
\end{description}
\begin{description}
	\item[A3]
	There exist constants $c_{1}, c_{2} > 0$ and a function $r_{2} \in C(\overline{\Omega})$ such that 
	$1 \le r_{2}(z) < p^{*}(z)$ and a function $k_{1} \in L^{1}(\Omega)$ such that for almost all 
	$z \in \Omega$ and for all $(s, \xi) \in \mathbb{R} \times \mathbb{R}^{N}$ we have
	$$\sum_{i = 1}^{N}a_{i}(z, s, \xi)\xi_{i} \ge c_{1}|\xi|^{p(z)} - c_{2}|s|^{r_{2}(z)} - |k_{1}(z)|$$
\end{description}
Let us consider the operator $A:  W^{1, p(\cdot)}(\Omega) \to  W^{1, p(\cdot)}(\Omega)^{*}$ given by
$$\langle A(u), v \rangle = \int_{\Omega}(a(z, u(z), \nabla u(z)), \nabla v(z))_{\mathbb{R}^{N}}dz \quad \textrm{for all} \quad u, v \in W^{1, p(\cdot)}(\Omega).$$ 
For $f \in W^{1, p(\cdot)}(\Omega)^{*}$ let us consider the following problem
\begin{equation}\label{eq3}
Au = f.
\end{equation}
We can  obtain the theorem
\begin{theorem}
	If conditions A1 - A3 hold then
	\item[(i)](Existence) For every $b \in W^{1, p(\cdot)}(\Omega)^{*}$ the problem (\ref{eq3}) has a solution.
	\item[(ii)](the Galerkin method) For fixed $b \in W^{1, p(\cdot)}(\Omega)^{*}$ and for every $n \in \mathbb{N}$ the $n$-th step in the Galerkin method has a solution $u_{n}$. Then every subsequence $\{u_{n_{k}}\}_{k \in \mathbb{N}}$ which converges weakly converges also strongly to solution $u$ of the problem (\ref{eq3}). If equation (\ref{eq3}) has the only one solution then whole sequence $\{u_{n}\}_{n \in \mathbb{N}}$ converges strongly to solution $u$.
\end{theorem}
\begin{proof}
	Let us observe that the space $W^{1, p(\cdot)}(\Omega)$ is a real separable and reflexive Banach space with dim  $W^{1, p(\cdot)}(\Omega) = +\infty$ (see \cite{Di Ha Ha Ru} Theorem 8.1.6, p. 249). From Lemma \ref{A L lemma 2} we know that the operator $A$ is well defined, bounded, pseudomonotone and coercive. Moreover from Theorem \ref{Main theorem} we know that the operator $A$ satisfies S+ property. So the thesis is a direct application of the Brezis theorem (see \cite{Zeidler} Theorem 27.A, p. 589) 
\end{proof}

\section{Appendix}
Let $X$ be a normed space.
\begin{definition}\label{Appendix def 9}
	We define
	$$P_{kc}(X) = \{A \subset X: A \quad \textrm{is nonempty, compact, convex}\}$$
	$$P_{wkc}(X) = \{A \subset X: A \quad \textrm{is nonempty, weakly compact, convex}\}.$$
\end{definition}

\begin{definition}\label{Appendix def 10}
	By $L^{0}(\Omega, X)$ we denote the space of Lebesgue measurable functions.
\end{definition}
\begin{definition}\label{Appendix def 11}
	Let $(S, \Sigma, \mu)$ be $\sigma$-finite measure space and $X$ is a separable Banach space. For multifunction $F: \Omega \to 2^{X}\setminus\{\emptyset\}$ we define
	$$S_{F} = \{f \in L^{0}(S, X)\}: f(\omega) \in F(\omega) \quad \textrm{for} \quad \mu-\textrm{almost all} \quad \omega \in S\}$$
	and for $p \in [1, +\infty]$ we define
	$$S_{F}^{p} = S_{F}\cap L^{p}(S, X).$$
\end{definition}
\begin{definition}\label{Appendix def 12}
	The first multivalued term in  (\ref{I eq3}) is interpreted as
	$$\{\divv v \in L^{p^{'}}(\Omega, \mathbb{R}^{N}), v(z) \in a(z, x(z), \nabla x(z)) \quad \textrm{a. a.}\quad z \in \Omega\}.$$
\end{definition}
\begin{definition}\label{Appendix def 13}
	By solution of problem (\ref{I eq3}) we understand $x \in W^{1, p}_{0}(\Omega)$ such that there exist $v \in S^{p^{'}}_{a(\cdot), x(\cdot), \nabla x(\cdot)}$ and $u \in S^{p^{'}}_{\partial j(\cdot, x(\cdot))}$ such that
	$$- \divv v(z) - u(z) = f(z, x(z), \nabla x(z)) \quad \textrm{for a.a} \quad z \in \Omega.$$ 		
\end{definition}
Let $Y$ be a normed space and $F:X \to 2^{Y}$ be a multifunction. For a definition of a graph of multifinction $F$, see definition 1.2.2 in \cite{Gas Papa}. For a definition of upper and lower semicontinuity of multifunction $F$, see definition 1.2.1 a) and b) in \cite{Gas Papa}. For a definition of graph measurability of multifunction $F$, see definition 1.2.6 b) in \cite{Gas Papa}.

Let $A: X \to X^{*}$ be an operator. For a definition of bounded operator, see. For a definition of coercive operator, see definition 2.5 in \cite{Roubicek}. For a definition of pseudomonotone operator, see definition 2.1 iv) in \cite{Roubicek}.

Let $A: X \to 2^{X^{*}}$ be a multivalued operator. For a definition of strictly monotone operator $A$, see definition 1.4.3 b) in \cite{Gas Papa}. For a definition of generalized pseudomonotone operator $A$, see definition 1.4.8 b) in \cite{Gas Papa}. For a definition of domain of operator $A$, see section 1.4.2 in \cite{Gas Papa}.

Let function $\phi: X \to \mathbb{R}$ be proper and convex. For a definition of subdifferential of function $\phi$ at point $x \in X$, see definition 1.3.3 in \cite{Gas Papa}.

\end{document}